\documentclass{amsart}

\DeclareSymbolFont{bbold}{U}{bbold}{m}{n}
\DeclareSymbolFontAlphabet{\mathbbold}{bbold}

\usepackage{amsmath}%

\usepackage{amsfonts}%
\usepackage{amssymb}%
\usepackage{graphicx}
\usepackage{bm}

\usepackage{mathabx}
\usepackage{rotating}
\usepackage[all,cmtip]{xy}
\usepackage{cite}
\usepackage{hyperref}
\usepackage{geometry}
\usepackage{bbm}
\usepackage{stmaryrd}
\usepackage{tensor}
\usepackage{wasysym}

\newtheorem{dummy}{dummy}[section]              
\newtheorem{lemma}[dummy]{Lemma}
\newtheorem{theorem}[dummy]{Theorem}
\newtheorem{corollary}[dummy]{Corollary}
\newtheorem{proposition}[dummy]{Proposition}

\newtheorem{conjecture}[dummy]{Conjecture}

\theoremstyle{definition}                                  
\newtheorem{definition}[dummy]{Definition}
\newtheorem{example}[dummy]{Example}
\newtheorem{remark}[dummy]{Remark}

\newcommand{\Bord}{\mathbf{Bord}}

\newcommand{\Alg}{\mathbf{Alg}}
\newcommand{\Vect}{\mathbf{Vect}}

\DeclareMathOperator{\Aut }{Aut}

\DeclareMathOperator{\Hom}{Hom}

\DeclareMathOperator{\colim}{colim}

\newcommand{\Mod}{\mathbf{Mod}}
\newcommand{\module}{-\mathrm{mod}}

\newcommand{\QC}{QC}

\DeclareMathOperator{\Map}{Map}

\DeclareMathOperator{\Rep}{Rep}

\newcommand{\Cat}{\mathbf{Cat}}

\newcommand{\Tot}{\mathbf{Tot}}

\newcommand{\bs}{\backslash}

\newcommand{\HC}{HC}

\newcommand{\Loc}{\mathcal Loc}

\newcommand{\fg}{\mathfrak{g}}
\newcommand{\fgx}{\mathfrak{g}^\star}

\newcommand{\cA}{\mathcal A}
\newcommand{\cB}{\mathcal B}
\newcommand{\cC}{\mathcal C}
\newcommand{\cD}{\mathcal D}
\newcommand{\cE}{\mathcal E}
\newcommand{\cF}{\mathcal F}

\newcommand{\cL}{\mathcal L}
\newcommand{\cM}{\mathcal M}
\newcommand{\cN}{\mathcal N}
\newcommand{\cO}{\mathcal O}
\newcommand{\cP}{\mathcal P}

\newcommand{\cW}{\mathcal W}
\newcommand{\cX}{\mathcal X}

\newcommand{\cZ}{\mathcal Z}

\newcommand{\A}{\mathbb A}

\newcommand{\C}{\mathbb C}
\newcommand{\D}{\mathbb D}

\newcommand{\G}{\mathbb G}

\newcommand{\Z}{\mathbb Z}

\newcommand{\GGG}{G\backslash G_{dR}/G}

\newcommand{\ot}{\otimes}
\newcommand{\oo}{\infty}

\newcommand{\CC}{\mathbb C}
\newcommand{\Tr}{\mathcal Tr}

\makeatletter
\newcommand*\leftdash{\rotatebox[origin=c]{-45}{$\dabar@\dabar@\dabar@$}}
\newcommand*\rightdash{\rotatebox[origin=c]{45}{$\dabar@\dabar@\dabar@$}}
\makeatother

\newcommand{\quot}[3]{{#1}\backslash{#2}/{#3}} 

\newcommand{\Gm}{{\G_m}}

\calclayout

\newcommand{\adjquot}{{/_{\hspace{-0.2em}ad}\hspace{0.1em}}}

\title{The Character Field Theory and Homology of Character Varieties}

\author{David Ben-Zvi} \address{Department of Mathematics\\University
  of Texas\\Austin, TX 78712-0257} \email{benzvi@math.utexas.edu}
\author{Sam Gunningham}\address{Department of Mathematics\\University
  of Texas\\Austin, TX 78712-0257} \email{gunningham@math.utexas.edu}
\author{David Nadler} \address{Department of Mathematics\\University
  of California\\Berkeley, CA 94720-3840}
\email{nadler@math.berkeley.edu}

\begin{document}

\begin{abstract}
We construct an extended oriented $(2+\epsilon)$-dimensional topological field theory, the character field theory $\cX_G$ attached to a affine algebraic group in characteristic zero, which calculates the homology of character varieties of surfaces. It is a model for a dimensional reduction of Kapustin-Witten theory ($\cN=4$ $d=4$ super-Yang-Mills in the GL twist), and a universal version of the unipotent character field theory introduced in arXiv:0904.1247.
Boundary conditions in $\cX_G$ are given by quantum Hamiltonian $G$-spaces, as captured by de Rham (or strong) $G$-categories, i.e., module categories for the monoidal dg category $\cD(G)$ of $\cD$-modules on $G$.
We show that the circle integral $\cX_G(S^1)$ (the center and trace of $\cD(G)$) is identified with the category $\cD(G/G)$ of ``class $\cD$-modules", while for an oriented surface $S$ (with arbitrary decorations at punctures) we show that $\cX_G(S)\simeq{\rm H}_*^{BM}(Loc_G(S))$ is the Borel-Moore homology of the corresponding character stack. We also describe the ``Hodge filtration" on the character theory, a one parameter degeneration to a TFT whose boundary conditions are given by classical Hamiltonian $G$-spaces, and which encodes a variant of the Hodge filtration on character varieties. \end{abstract}
\maketitle

\section{Introduction}
Let us fix a field $k$ of characteristic zero and an affine algebraic group $G$ over $k$ -- for example a complex reductive group.

\subsection{Summary of results}
Given a topological surface $S$, the {\em character stack} (or Betti space) $\Loc_G(S)$ is the (derived) stack of $G$-local systems on $S$, i.e. of representations of the fundamental ($\infty$-)groupoid of $S$ into $G$. In particular for a pointed connected oriented surface of positive genus $g$, $\Loc_G(S)$ is the quotient by $G$ of the (derived) fiber of the map $G^{2g}\to G$ given by $$(A_1,B_1,...A_g,B_g)\mapsto \Pi_i A_iB_iA_i^{-1} B_i^{-1}.$$

We denote by $\Cat$ the $(\infty,1)$-category of $k$-linear stable presentable $\infty$-categories (or equivalently differential graded categories) where morphisms are colimit-preserving functors. We denote by $\Alg_{(1)}\Cat$ the Morita $(\infty,2)$-category of algebra objects in $\Cat$ (i.e., monoidal dg categories). (See Section~\ref{prelim} for the relevant definitions.)

\begin{theorem}\label{main}
There is a 2d oriented TFT, i.e., symmetric monoidal functor
\[
\cX_G: \Bord_2^{or} \to \Alg_{(1)}(\Cat),
\]
which assigns the following invariants:
\begin{itemize}
\item To a point, $\cX_G$ assigns the ``categorical group algebra'' $\cD(G)$ of $\cD$-modules on $G$ under convolution.
\item To a circle, $\cX_G$ assigns the category $\cD(G)^G$ of $G$-equivariant $\cD$-modules (``class sheaves") on $G$.
\item To a closed oriented surface $S$, $\cX_G$ assigns the Borel-Moore chains $C_\ast^{BM}(\Loc_G(S))$.
\item More generally to any oriented surface with boundary $S,\partial S= \partial_{in} S\coprod \partial_{out} S$ , $\cX_G$ assigns a functor $\cD(\Loc_G(\partial_{in} S))\to \cD(\Loc_G(\partial_{out} S)$ which is identified with the integral transform $q_*p^!$ along the correspondence 
$$\xymatrix{&\ar[dl]_-p \Loc_G(S) \ar[dr]^-q& \\
\Loc_G(\partial_{in} S)&&\Loc_G(\partial_{out}(S))}$$
\end{itemize}
\end{theorem}
 In this context, the Borel-Moore chains $C_\ast^{BM}(\Loc_G(\Sigma))$ means the $G$-equivariant Borel-Moore chains of $\Hom(\pi_1(\Sigma, x), G)$. The categorical group algebra $\cD(G)$ is Morita equivalent to the monoidal category $\HC_G$ of Harish-Chandra bimodules \cite{dario}, as well as (for $G$ reductive) to the universal Hecke categories $\cD(\quot NGN)$ and $\cD_{(H)}(\quot NGN)_{(H)}$ \cite{morita}; any of these monoidal categories may be considered as the value of $\cX_G$ on a point\footnote{We will use the notations $(-)^{(G)}, (-)_{(G)}$ for a group $G$ to denote weak $G$-(co)invariants, reserving $(-)^G, (-)_{G}$ for strong (co)invariants.}.
 
 \subsubsection{Twisted Character Stacks} Given a finite subset $\{x_i\}$ of $S$ labeled by conjugacy classes $\{\gamma_i\}$ in $G$ ($i=1, \ldots k$), we likewise consider the stack $\Loc_G(S,x_i, \gamma_i)$ of local systems on $S\setminus \{x_i\}$ with monodromy around $x_i$ contained in $\gamma_i$. The TFT $\cX_G$ encodes the Borel-Moore chains of $\Loc_G(S,x_i,\gamma_i)$ as follows. The punctured surface, considered as a cobordism from $k$ circles to the empty $1$-manifold, is assigned a functor from $\left(\cD(G)^G\right)^{\otimes k}$ to $\Vect$. Thus $\cX_G$ assigns a chain complex to a punctured surface $(S,x_i)$ where each puncture is decorated with with an object $\cF_i \in \cD(G)^G$. Taking the $\cF_i = \delta_{\gamma_i}$, the $D$-module of $\delta$-functions on $\gamma_i$, we obtain the homology of the twisted character stack.

\begin{example}
In the case $G=SL_n$, consider the case where the surface has a single puncture, decorated with a primitive $nth$-root of the identity matrix. This stack is represented by a $\Z/n\Z$-gerbe over a smooth affine variety, the twisted character variety appearing in the work of Hausel, Letellier and Rodriguez-Villegas \cite{HRV, HLRV}. The finite gerbe does not affect the cohomology, thus the values of the TFT $\cX_{SL_n}$ encode the cohomology groups of the twisted character varieties studied in loc. cit. Similarly, the values of the twisted $GL_n$ character variety are encoded by $\cX_{GL_n}$, where they appear with an extra factor consisting of the $Z(GL_n)$-equivariant cohomology of a point.
\end{example}

\subsubsection{Classical Limit and Hodge Filtration}
Recall that the classical Rees construction identifies filtered modules with flat $\Gm$-equivariant quasicoherent sheaves over $\A^1$, recovering our module as the fiber at $1$ and its associated graded as the fiber at $0$. For example the filtrations on enveloping algebras $U\fg$ and differential operators $\cD_X$ are embodies by the corresponding Rees constructions, which are $\Gm$-equivariant algebras $U_{\hbar}\fg$ and $\cD_{\hbar,X}$ over $\A^1$ with special fiber the associated graded algebras $Sym\fg$ and $\cO(T^*X)$.
In the nonlinear setting following~\cite{simpson} (see~\cite[IV.5.1]{GR} in the derived setting) one may define a filtration on an object to be a $\Gm$-equivariant family over $\A^1$ recovering our object as the fiber at $1$.

Categories of $\cD$-modules have a natural filtration and resulting degeneration to their classical analog, in which $\cD$-modules are replaced by quasicoherent sheaves on cotangent stacks. In particular the monoidal category $\cD(G)$ has a classical analog, the category $\QC(T^*G)$ with monoidal structure given by convolution: we consider $T^*G$, with its two (left and right) projections to $\fgx$, as a groupoid over $\fgx$ (the action groupoid for the coadjoint action of $G$ ), so that quasi coherent sheaves $\QC(T^*G)$ acquires a convolution monoidal structure.  Just as $\cD(G)$ is Morita equivalent to $\HC_G$, Gaitsgory's 1-affineness theorem~\cite{1affine} implies that $\QC(T^*G)$ is Morita equivalent to the rigid symmetric monoidal category $$\HC_{0,G}=\QC(\fgx/G)\simeq \QC(G\backslash T^*G/G),$$ which defines an oriented 2d topological field theory $\cX_{0,G}$ (as considered e.g. in~\cite{BFN}). 
The monoidal categories $\QC(T^*G)$ and $\HC_{0,G}$ are naturally graded, i.e., have a structure of algebras over $\QC(\G_m)$ (or of sheaves of monoidal categories over $pt/\G_m$), and the TFT $\cX_{0,G}$ likewise inherits a natural grading.
 
We can interpolate between $\cD(G)$ and $\QC(T^*G)$ (or between $\HC_G$ and $\HC_{0,G}$) using the Rees construction for the standard filtration on $\cD(G)$ (equivalently by considering sheaves on Simpson's Hodge filtration $G_{Hod}$ on $G_{dR}$~\cite{simpson}). The result is a filtered monoidal category $\cD_{\hbar}(G)\in\Cat_{filt}$ -- i.e., a $\G_m$ equivariant family of monoidal categories $\cD_\hbar(G)$ for $\hbar\in \A^1$ (a family of monoidal categories over $\A^1/\G_m$, or by~\cite{1affine} an algebra over $\QC(\A^1/\G_m)$).
The $\oo$-categorical theory of the filtration on $\cD$-modules and the Rees construction is developed in~\cite[IV.4-5]{GR}.

Theorem \ref{main} can be generalized as follows:

\begin{theorem}\label{filtered main}
There is a filtered 2d oriented TFT 
\[
\cX_{\hbar,G}: \Bord_2^{or} \to \Alg_{(1)}(\Cat_{filt}),
\]
with $\cX_{1,G}\simeq \cX_G$ and the associated graded theory $\cX_{0,G}$ has the following values:
\begin{itemize}
\item To a point, $\cX_{0,G}$ assigns the monoidal category $\QC(\fgx/G)$.
\item To a circle, $\cX_{0,G}$ assigns $\QC(T^*(G/G))$.
\item To a closed oriented surface $S$, $\cX_{0,G}$ assigns the Dolbeault cohomology of the character stack,
$$\Gamma(T*[-1]\Loc_G(S),\cO)=\Gamma(\Loc_G(S),\Omega^\cdot).$$ More generally the value on an oriented surface with boundary is given by integral transforms along $\Loc_G(S)$.
\end{itemize}
Moreover the resulting filtration on $\cX_{1,G}(S)$ coincides with the Hodge filtration on cohomology of the dualizing complex of $\Loc_G(S)$.  
\end{theorem}
The filtration on the Borel-Moore homology of $\Loc_G(S)$ that arises is the Hodge filtration studied in~\cite[IV.5.5]{GR}, the derived version of the so-called ``naive'' Hodge filtration, i.e. the filtration arising from the de Rham to Dolbeault degeneration of the (non-proper, and not necessarily smooth) character stack (a form of the Hodge filtration of cyclic homology~\cite{Weibel}). As remarked upon in \cite{Weibel,SiT}, the naive Hodge filtration filtration on a (non-smooth or non-proper) variety does not always agree with the Deligne Hodge filtration from the theory of mixed Hodge structures. Nonetheless, we expect that for character stacks these filtrations do, indeed, agree:
\begin{conjecture}\label{naive}
	The naive Hodge filtration on the Borel-Moore homology of character stacks for complex reductive $G$ agrees with the Hodge filtration in the sense of mixed Hodge structure. 
\end{conjecture}

For a mixed Hodge structure of Tate type, the Hodge filtration and the weight filtration agree up to reindexing. Thus, Conjecture \ref{naive} implies that the mixed Hodge polynomials of character stacks of complex reductive groups are encoded in the TFT $\cX_{\hbar,G}$.

\begin{remark}[Filtered proofs]
The proof of Theorem~\ref{filtered main} closely imitates the proof of Theorem~\ref{main} but is notationally more cumbersome and the background less documented in the literature.
We therefore confine ourself throughout the text to proofs in the unfiltered setting and accompanying remarks explaining the necessary modifications for the filtered version.
\end{remark}

\subsection{Motivation: Character varieties}\label{ss character varieties}
The cohomology of character varieties of complex reductive groups is a subject of great interest and extremely rich structure, see in particular~\cite{HRV,HLRV,dCHM}. In~\cite{HRV} the cohomology of character varieties is studied using the counts of points of character varieties of the finite Lie groups $G(F)$ for $F$ a finite field. Such counts are well known to fit in the framework of topological field theory. Namely for any finite group $\Gamma$ there is an extended 2d topological field theory $Z_\Gamma$ (topological $\Gamma$-Yang-Mills theory, see for example~\cite{Freed,FHLT}), which assigns to any closed oriented surface $S$ the count of $\Gamma$-bundles
$$
Z_\Gamma(\Sigma)=\# \{\Hom(\pi_1(\Sigma),\Gamma)/\Gamma\}\in \C 
$$ 
where as usual a bundle $\cP$ is weighted by $1/\Aut(\cP)$. This is the volume of  
the orbifold of $\Gamma$-bundles on the surface. To a circle, $Z_\Gamma$ assigns 
class functions:
$$
Z_\Gamma (S^1)=\C [\Gamma]^\Gamma=\C[\Gamma\adjquot {\Gamma}] \in \Vect^{fd}_\C,$$
the vector space of functions on the orbifold of $\Gamma$-bundles on the circle.
To a point, $Z_\Gamma$ assigns the category of finite-dimensional representations of $\Gamma$ (or $\C[\Gamma]$-modules)
$$
Z_\Gamma(pt) = \Rep^{fd}_\C(\Gamma) \in AbCat_\C.
$$
This is the category of finite-dimensional algebraic vector bundles on the
orbifold of $\Gamma$-bundles on a point. Equivalently we may consider $Z_\Gamma$ as taking values in the Morita 2-category of associative algebras, and assigning $Z_\Gamma(pt)=\C[\Gamma]$. The value of a field theory $Z$ on a point encodes the category of boundary conditions of $Z$ (with morphisms given by interfaces). Thus boundary conditions for $Z_\Gamma$ are given by representations of $\Gamma$ -- e.g., the space of functions $\C[X]$ on any finite $\Gamma$-set $X$. By the Cobordism Hypothesis~\cite{jacobTFT} this assignment in turn recovers the rest of the structure of the field theory.

In this paper we show that the Borel-Moore homologies of the complex character stacks $\Loc_G(S)$ are themselves the output on $S$ of an extended topological field theory $\cX_G$, a categorified analog of the finite group theory $Z_\Gamma$. The role of functions on finite sets in the construction of $Z_\Gamma$ is taken up by $\cD$-modules on varieties and stacks in the construction of $\cX_G$. In particular the value $\cX_G(S^1)$ on the circle is given by ``class $\cD$-modules", i.e. adjoint equivariant $\cD$-modules on the group $G$, or $\cD$-modules on $\Loc_G(S^1)\simeq G/G$. 

There is likewise a natural categorification of the group algebra $\C[\Gamma]$, namely the monoidal category $\cD(G)$ of $\cD$-modules on the group $G$ under convolution. In order to apply the cobordism hypothesis and construct a topological field theory out of a monoidal category $\cC$, we need it to satisfy a strong finiteness condition, {\em 2-dualizability}, and for the resulting TFT to be defined on all oriented surfaces we further need to equip $\cC$ with a nondegenerate cyclic trace - a {\em Calabi-Yau structure}.
A natural finiteness condition to ask for a monoidal category is that it be {\em rigid}, i.e. that [compact] objects have left and right duals. Indeed in~\cite{character} it is shown that rigid monoidal categories are 2-dualizable. This result was applied to construct the {\em unipotent character theory} of a reductive group $G$ starting from the Hecke category $\cD(\quot BGB)$. However $\cD(G)$ is not rigid when $G$ is not finite, and thus it was not clear (and indeed was not expected by experts) that $\cD(G)$ should be 2-dualizable. 

The construction of an extended field theory from a monoidal category $\cC$ however depends on $\cC$ only up to Morita equivalence (i.e., only its 2-category $\cC\module$ of module categories). As pointed out in~\cite{dario}, Gaitsgory's fundamental 1-affineness theorem~\cite{1affine} implies that $\cD(G)$ is Morita equivalent to the {\em rigid} monoidal category $\HC_G$ of Harish-Chandra bimodules (and hence is 2-dualizable). Recall that $$\HC_G=\cD_{(G)}(G)_{(G)}=\Mod_{\fg}^{(G)}=\Mod_{\fg\ot\fg}^G$$ is the convolution category of left- and right-weakly $G$-equivariant $\cD$-modules on $G$ -- or more concretely, $U\fg$-bimodules which are integrable to $G$ under the diagonal action. We further prove the following:

\begin{theorem}
The monoidal category $\cD(G)$ is a fully dualizable Calabi-Yau object of $\Alg_{(1)}(\Cat)$.
\end{theorem}

\begin{remark}
	As mentioned, the categorical group algebra $\cD(G)$ is (monoidally) Morita equivalent to $\HC_G$; it is also Morita equivalent (for $G$ reductive) to the universal Hecke category $\cD_{(H)}({\quot{N}{G}{N}})_{(H)}$ (these are both rigid monoidal). Each of these categories thus carries the structure of a Calabi-Yau algebra object.
	One can equally well consider $\cX_G$ as assigning any of these monoidal categories to a point. Alternatively, one can instead assign their categories of modules in $\Cat$ (also known as de Rham $G$-categories). 
\end{remark}

\subsection{Motivation: Boundary conditions and Hamiltonian actions}\label{quantum}
The cobordism hypothesis may be interpreted as the assertion that extended topological field theories are completely determined by the higher category formed by boundary conditions and their interfaces. For example in two dimensions a topological gauge theory with gauge group $G$ is determined by specifying the class of $G$-quantum mechanics that form boundary conditions, while in three dimensions we need to specify a class of two-dimensional field theories with $G$ symmetry. In our current setting, where we attach a monoidal category $\cA$ to a point, boundary conditions form the 2-category of $\cA$-module categories. Thus by specifying different categorical ``group algebras" (monoidal categories of sheaves on $G$ with convolution) we can construct different gauge theories. Three natural examples are given by local systems, Higgs sheaves and $\cD$-modules on $G$, leading to theories of Betti, Dolbeault and de Rham $G$-categories.

First let us mention the theory of Betti (or locally trivial) $G$-categories whose theory is beautifully developed by Teleman~\cite{teleman}. There we take for $\cA$ the monoidal category of local systems on $G$, i.e., sheaves on the homotopy type of $G$, or modules for the $\cE_2$-algebra of chains on the based loop space of $G$. Examples of module categories are provided by Fukaya categories of spaces with Hamiltonian actions of the maximal compact $G_c$ of $G$. In other words, boundary conditions come from equivariant $A$-models. This provides a mathematical model of the ``A-twist" of 3d $\cN=4$ super-Yang-Mills theory, for which Teleman develops complete ``character formulae". 

The classical character theory $\cX_{0,G}$ has as boundary conditions {\em Dolbeault $G$-categories}, or $B$-models of algebraic Hamiltonian $G$-spaces. Consider an Hamiltonian $G$-space, that is, a Poisson $G$-variety $X$ (with Poisson ring of functions $A_0=\CC[X]$) equipped with a $G$-equivariant, Poisson moment map $\mu:X\to \fgx$ (i.e. $\CC[\fgx]\to A_0$), or equivalently $$\underline{\mu}:X/G\to \fgx/G,$$ which generates the $\fg$-action on $X$. Standard examples are provided by $X=T^*M$ for $M$ a $G$-variety and $X=\fgx$ itself. 
This data can be reformulated as a Poisson map $X\to \fgx$ equipped with an 
action of $T^*G$ (as groupoid over $\fgx$) on $X$. This endows $\QC(X)$ with the structure of module category for the monoidal category $\QC(T^*G)$ of Higgs sheaves on $G$. 
Equivalently, the equivariant moment map $\underline{\mu}$ endows $\QC(X/G)$ with the structure of $\HC_{0,G}=\QC(\fgx/G)$-module category, which corresponds to $\QC(X)$ under the Morita equivalence between $\HC_{0,G}$ and $\QC(T^*G)$. 

The character theory $\cX_G$ has as boundary conditions module categories $\cD(G)\module$ for $\cD(G)$, known as de Rham, or strong, $G$-categories, whose study plays a central role in geometric representation theory and the geometric Langlands correspondence~\cite{BDHecke}. They arise from quantum analogs of Hamiltonian $G$-spaces -- e.g., deformation quantized $B$-models of Hamiltonian $G$-spaces. 
Such an analog is provided by a strong or Harish-Chandra action of $G$ on an associative algebra $A$: i.e., an action of $G$ on $A$ by algebra automorphisms, for which the Lie algebra action is made internal by means of a homomorphism $\mu^*:U\fg\to A$. Standard examples are the ring of differential operators $A=\Gamma(\cD_M)$ for a $G$-space $M$ and $A=U\fg$ itself. In this case the action of $G$ on the category $A\module$ is enhanced to the structure of module category for the monoidal category $\cD(G)$. 
Equivalently, applying the Morita equivalence between $\cD(G)$ and $\HC_G$, the weak $G$-invariants $(A\module)^{(G)}$ form a module category for $\HC_G$. 

\subsection{Supersymmetric gauge theory context}
The character field theory $\cX_{G}$ (and its deformation $\cX_{\hbar,G}$) is a mathematical model (as an extended TFT) for a ``B-type" topological twisted form of
maximally supersymmetric ($\cN=8$) three-dimensional gauge theory with gauge group the compact form $G_c$ of $G$. Namely, 3d $\cN=8$ super Yang-Mills theory is a dimensional reduction of 4d $\cN=4$ super Yang-Mills, and we consider a dimensionally reduced analog of the ``$\widehat{B}$-model"~\cite{WittenAtiyah} -- the Kapustin-Witten twist with the B-type supercharge (corresponding to the parameter $\Psi=\infty$). In 3d language, we are considering $\cN=4$ super-Yang-Mills with (massless) adjoint matter. After the topological twist the adjoint scalars become adjoint one-forms, which we add to the gauge field to define a complex $G$-connection. The graded parameter $\hbar$ arises as a Nekrasov $\epsilon$-parameter: it comes as the 3d limit of compactifying 4d $\cN=4$ on a circle with an $\Omega$-background around the circle. In other words, we consider a supercharge $Q$ that squares to a translation along the compactification direction~\cite{WittenAtiyah}. This deformation does not affect the underlying supersymmetric field theory~\cite{NekrasovWitten} but changes the topological theory. 

The Kapustin-Witten $\widehat{B}$-model compactified on a circle times a surface $S$ produces the derived functions on $\Loc_G(S\times S^1)$, the inertia stack (or derived loop space) of the character stack. Our dimensionally reduced theory $\cX_G(S)$ produces differential forms (or polyvector fields) on the character stack, which is the linearization of functions on the loop space along constant loops.

The identification of the character theory as a dimensional reduction of the Kapustin-Witten $\widehat{B}$-model can also be seen directly from our assignment of the 2-category $\Mod_{\cD_{\hbar}(G)}\simeq \Mod_{\HC_{\hbar,G}}$ as the boundary conditions in $\cX_{\hbar,G}$. Recall that boundary conditions in a circle compactification of a TFT $\cZ$ may be interpreted as codimension 2 defects of $\cZ$.
Since the category $\HC_{0,G}$ is known~\cite{bezfink} to be identified with the derived Satake category, i.e., with line operators (or endomorphisms of the trivial surface defect) in the Kapustin-Witten theory, our boundary conditions $\Mod_{\HC_{0,G}}$ are identified with surface defects in the 4d theory which have nontrivial interface with the trivial surface defect, a topological way of imposing that the compactification circle is taken to be small.

\subsection{Acknowledgments}
We would like to acknowledge the National Science Foundation for its support through individual grants DMS-1103525 (DBZ) and DMS-1502178 (DN). We would like to thank Chris Schommer-Pries and Hendrik Orem for helpful discussions of $(1,2)$-TFTs and idempotents in the character theory, respectively.

\section{Calabi-Yau structure on the categorical group algebra}

The goal of this section is to prove the following result.
\begin{theorem}\label{thm-cy}
	The monoidal category $\cD(G)$ carries the structure of a Calabi-Yau algebra object in $\Cat$.  
 
\end{theorem}
This is analogous to the claim that the group algebra $\C\Gamma$ of a finite group $\Gamma$ is a symmetric Frobenius algebra. Recall that, as $\cD(G)$ is Morita equivalent to the rigid monoidal category $\HC_G$ ~\cite{dario}, it is also fully dualizable as an object of $\Alg_{(1)}(\Cat)$. Thus, Theorem~\ref{thm-cy} together with Lurie's cobordism hypothesis~\cite{jacobTFT}, shows that $\cD(G)$ defines a 2d TFT $\cX_G$ valued in $\Alg_{(1)}(\Cat)$. 

We also describe a filtered version of the TFT $\cX_G$, in which $\cD(G)$ is replaced by its Rees construction $\cD_\hbar(G)$:

\begin{theorem}
The filtered monoidal category $\cD_\hbar(G)$ defined by the Rees construction on $\cD(G)$ carries the structure of a Calabi-Yau algebra object in $\Cat_{filt}$. 
\end{theorem}
Likewise, $\cD_\hbar(G)$ is Morita equivalent to a rigid filtered monoidal category $\HC_{\hbar,G}$, the Rees construction on $\HC_G$, and so defines a topological field theory $\cX_{\hbar,G}$ valued in the Morita theory of filtered monoidal categories, $\Alg_{(1)}(\Cat_{filt})$.

\subsection{Preliminaries}\label{prelim}
We will work over a fixed field $k$ of characteristic zero. Let $\Cat$ denote the $(\infty,1)$ category of $k$-linear, stable, presentable $\infty$-categories where morphisms are functors which are left adjoints (equivalently, functors which preserve small colimits). Recall that $\Cat$ comes equipped with a symmetric monoidal structure $\otimes$. The unit object of $\Cat$ is the $\infty$-category $\Vect_k$ of (differential graded) $k$-vector spaces. We denote by $\Alg(\Cat)$ the $(\infty,1)$-category of algebra objects in $\Cat$, i.e., monoidal dg categories, and by $\Alg_{(1)}(\Cat)$ the Morita $(\infty,2)$-category of monoidal categories~\cite{jacobTFT}. In other words the objects of $\Alg_{(1)}(\Cat)$ are algebra objects in $\Cat$ and the morphism $(\infty,1)$-categories are given by bimodules, $Map_{\Alg_{(1)}(\Cat)}(\cA,\cB)=\Mod_{\cA\ot\cB^{op}}(\Cat)$. (Note that we will not need the theory of symmetric monoidal $(\infty,2)$-categories, other than to quote the cobordism hypothesis and verify its explicit conditions.)

Likewise we consider a filtered version, in which $\Vect$ is replaced by the rigid symmetric monoidal category $\QC(\A^1/\Gm)$ in $\Cat$. Indeed as explained in~\cite[IV.5.1]{GR}, for a category $\cC\in\Cat$, filtered objects of $\cC$ (i.e., functors from $\Z$ to $\cC$) are identified with objects of $(\cC\otimes\QC(\A^1))^\Gm$. Following~\cite{simpson} we take as the definition of a filtration a $\Gm$-equivariant family over $\A^1$, and so define $\Cat_{filt}:=\Cat_{\A^1/\G_m}$, the symmetric monoidal category of $\QC(\A^1/\Gm)$-modules in $\Cat$, which by~\cite{1affine} is equivalent to quasicoherent sheaves of categories over $\A^1/\G_m$. 
The symmetric monoidal category $\Cat_{\A^1/\G_m}$ is in particular ``good" in the sense of~\cite[Definition 4.1.7]{jacobTFT}, i.e. admits sifted colimits, which are preserved by the monoidal structure. We thus have filtered monoidal categories $\Alg(\Cat_{filt})$ and their Morita theory $\Alg_{(1)}(\Cat_{filt})$, sharing the formal properties of their unfiltered version.

Given an algebra object $\cA$ in $\Cat$ (or any good symmetric monoidal category, such as $\Cat_{filt}$), recall~\cite{jacobTFT} that a \emph{Calabi-Yau} structure on $\cA$ is the data of an $S^1$-equivariant morphism
\[
t: \Tr(\cA) \to \Vect
\]
such that the composite morphism
\[
\cA \times \cA \xrightarrow{\ast} \cA \to \Tr(\cA) \xrightarrow{t} \Vect
\]
is the evaluation map exhibiting $\cA$ as self dual in $\Cat$.

\subsection{Cyclic Bar Construction}\label{cyclic section}
Given an algebra object $\cA$ in $\Cat$
the \emph{bar resolution} of $\cA$ is the simplicial object $\cA^{\bullet +2}$ of $\Cat$, where the face maps are given by the monoidal operations and the degeneracies are given by unit maps (see e.g.~\cite{BFN}). The bar resolution is augmented by $\cA$, and the augmentation gives an equivalence
 $$
  \xymatrix{
  \cA \simeq 
  \left| \cA^{\otimes (\bullet+2)} \right| 
  }
  $$
  where $| \quad |$ denotes the colimit (also called geometric realization) of the corresponding diagram in $\Cat$. The bar construction can be used to calculate the \emph{center}
  \begin{align*}
\cZ(\cA) =  \Hom_{\cA \otimes \cA^{op}}(\cA, \cA) &= \Hom_{\cA \otimes \cA^{op}}(\left| \cA^{\otimes (\bullet+2)} \right|,  \cA)\\
 &=  \Tot\left\{\Hom_{} (\cA^{\otimes \bullet},  \cA) \right\} \\
\end{align*}
and \emph{trace}
\begin{align*}
 \Tr(\cA) = \cA \otimes_{\cA \otimes \cA^{op}} \cA &= \left| \cA^{\otimes (\bullet+2)} \right| \otimes_{\cA \otimes \cA^{op}} \cA \\
  &= \left|\cA^{\otimes (\bullet+1)} \right|
\end{align*}

In order to express the cyclic symmetry of the trace, it is convenient to take a more geometric model as factorization or topological chiral homology $Tr(\cA)\simeq \int_{S^1}\cA$, which we now recall following ~\cite[Section 5.3.3]{l-ha} and~\cite{character}.

First, given a framed circle $S$ (for example, the standard circle $S^1=[0,1]/\hspace{-0.25em}\sim$ with its natural induced orientation) let $I_S$ be the $\infty$-category given by the nerve of the poset of framed embeddings of finite disjoint unions of open intervals in $S$ with partial order given by framed inclusions. We note that the $\infty$-category $I_S$ is equivalent to the opposite category of the \emph{cyclic category} $\Lambda$; an $I_S$ shaped diagram in a category is known as a \emph{cyclic object}.

Given a basepoint $0\in S$ we let $I'_{S,0}$ denote the nerve of the poset of framed embeddings of intervals, in which one interval is marked, and the marked interval must cover the basepoint. Note that $I'_{S,0}$ is equivalent to the opposite of the simplex category. There is a functor $I'_{S,0}\to I_S$, which forgets the marking; thus we say that every cyclic object ($I_S$-shaped diagram) has an underlying simplicial object ($I_{S}'$-shaped diagram), see~\cite[Lemma 5.3.3.10]{l-ha}.

Given an $E_1$-algebra object $\cA$ in a symmetric monoidal category $\cC$, we have a cyclic object given by the functor 
$$
\xymatrix{
\cA^{\otimes -}:I_S\ar[r] &  \cC
& \cA^{\otimes -}(I) = \cA^{\otimes \pi_0(I)}
}$$
whose structure on morphisms is given by the $E_1$-structure on $\cA$. By definition, the factorization (or topological chiral) homology of $\cA$ over $S$ is the colimit
$$
\xymatrix{
\int_{S} \cA= \colim_{I_S} \cA^{\otimes -}
}$$
Observe that  the above makes sense for families of framed circles, and so exhibits $\int_{S} \cA$ as the fiber of an $\oo$-local system  over
the moduli of framed circles $B{Diff^+}(S^1) = BS^1$ (or equivalently, $\int_{S^1} \cA$ carries an action of $S^1$). 

Choosing a basepoint of $0\in S$, we see that the underlying simplicial object of $\cA^{\otimes -}$ is the bar complex which computes the trace, and thus (by~\cite[Theorem 5.3.3.11]{l-ha})
\[
\int_{S} \cA = \colim_{I_S} \cA^{\otimes -} \xleftarrow{\sim} \colim_{I_{S,0}'} \cA^{\otimes -} = \Tr(\cA).
\]
In particular, $\Tr(\cA)$ carries a canonical $S^1$-action.

\subsection{$\cD$-modules and filtered $\cD$-modules}\label{filtered section}
The general theory of $\cD$-modules in derived algebraic geometry is developed in the book~\cite{GR} as an instance of the theory of ind-coherent sheaves on (ind-)inf-schemes. This is a powerful setting that extends the setting of derived schemes to include objects such as quotients by formal groupoids. An important example of an inf-scheme is Simpson's de Rham space $X_{dR}$~\cite{simpson} for a scheme $X$, which can be described as the quotient of $X$ by the formal groupoid associated to the tangent Lie algebroid, namely the formal completion of the diagonal of $X$. Gaitsgory and Rozenblyum define $\cD$-modules on $X$ as ind-coherent sheaves on $X_{dR}$ and deduce the standard functorial properties of $\cD$-modules from their general formalism for ind-coherent sheaves~\cite[III.4]{GR}. This definition can easily be compared to the traditional definition for $X$ a smooth classical scheme, by realizing sheaves on $X_{dR}$ as modules for the enveloping algebra $\cD_X$ of the tangent Lie algebroid.
 
Given a Lie algebra $\fg$ over $k$, we may obtain a Lie algebra $\fg_\hbar$ over $\A^1/\G_m$ by rescaling the Lie bracket by $\hbar\in \A^1$, ~\cite[IV.5.1]{GR}. The enveloping algebra $U\fg_\hbar$ is (in the classical setting) the Rees construction on $U\fg$, deforming $U\fg$ to the graded algebra $Sym\fg$. We can preform an analogous construction for Lie algebroids, in particular for the tangent Lie algebroid of a scheme $X$~\cite[IV.5.3]{GR}. This results in another example of an inf-scheme, Simpson's Hodge space $X_{Hod}\to \A^1/\G_m$, which embodies the Hodge filtration on (nonabelian) de Rham cohomology (see also~\cite{simpsongeometric}). It is defined as the quotient of $X$ over $\A^1/\G_m$ by the formal groupoid given by the deformation to the normal cone of the de Rham groupoid, corresponding to rescaling the Lie bracket on the tangent sheaf by $\hbar$. In other words, sheaves on the Hodge space (for $X$ a smooth classical scheme) are modules for the Rees construction for the filtered algebra $\cD_X$. The fiber over $\hbar=0$ is given by the category $\QC(T^*X)$ as a graded category (i.e., with its standard $\Gm$ action). We denote the resulting category $IndCoh(X_{Hod})$ by $\cD_\hbar(X)$ and refer to it informally as filtered $\cD$-modules.
The formalism of~\cite[III.4]{GR} for functoriality of $\cD$-modules carries over to the Hodge setting, once one observes that a map $f:X\to Y$ defines a {\em filtered} map of de Rham groupoids, hence a map $f_{Hod}:X_{Hod}\to Y_{Hod}$ which shares the good properties of $f_{dR}$ (in particular for $f$ nil-schematic, $f_{Hod}$ is inf-schematic). As a result one has the basic $\cD$-module formalism of $f_*,f^!$ and their duality in the filtered setting.

\subsection{Categorical (co)invariants}
Recall that $G$ denotes an affine algebraic group over $k$. Convolution defines monoidal structures on the categories $\QC(G)\simeq IndCoh(G)$, $\cD(G)$ and $\cD_\hbar(G)$ of quasicoherent sheaves, $\cD$-modules and filtered $\cD$-modules. We would like to compare modules for these monoidal categories with modules for $Rep(G)=\QC(G\bs G/G)=\Vect^{(G)}$, the category of Harish-Chandra bimodules $$\HC_G=\QC(G\bs G_{dR}/G)\simeq \Mod_{\fg}^{(G)}$$ and its filtered version $$\HC_{\hbar,G}=\QC(G\bs G_{Hod} /G)\simeq \Mod_{\fg_\hbar}^{(G)}.$$ 

Recall that if $\cC$ is a $G$-category (i.e. a $\cD(G)$-module), then the \emph{coinvariants} are defined by $\cC_G = \cC \otimes_{\cD(G)} \Vect$ and the invariants by $\cC^G = \Hom_{\cD(G)}(\Vect, \cC)$. Explicitly, the coinvariants are computed as the colimit of the simplicial diagram 
\begin{equation}\label{diagramcoinvariants}
\cC \leftleftarrows \cC\otimes \cD(G) \ldots
\end{equation}
and the invariants are computed as the limit of the cosimplicial diagram
\begin{equation}\label{diagraminvariants}
\cC \rightrightarrows \Hom(\cD(G),\cC) \ldots
\end{equation} 
Similar definitions hold for $\cC$ an algebraic $G$-category ($\QC(G)$-module), giving rise to the notion of weak $G$ (co)invariants $\cC^{(G)},\cC_{(G)}$  and for
 $\cC\in \Cat_{filt}$ a $\cD_\hbar(G)$-module, with the role of the $\cD(G)$-module $\Vect$ being taken by the unit $\QC(\A^1/\G_m)$,i.e., the augmentation module over $\cD_\hbar(G)$.

The key technical input to the proof of Theorem \ref{thm-cy} is Gaitsgory's 1-affineness theorem~\cite{1affine} for $pt/G$, which asserts that the functor of weak $G$-invariants defines an equivalence $$\Mod_{\QC(G)}\simeq \Mod_{Rep(G)}.$$
Thus algebraic $G$-categories are described as modules over the rigid symmetric monoidal category $Rep(G)$. We will utilize Beraldo's ensuing parallel story for $\cD(G)$-modules \cite{dario} identifying (strong) coinvariants and invariants for a group action on a category. We will use this to identify the center and trace of the monoidal category $\cD(G)$ (later we will need to keep track of extra structures on these identifications, so we take extra care to make sure everything can be understood explicitly).

\begin{theorem} \cite[Proposition 3.5.6, Theorem 3.5.7]{dario} \label{Morita thm} There is a natural equivalence
$$\Mod_{\cD(G)}\simeq\Mod_{\HC_G}$$ between $\cD(G)$ and the rigid monoidal category $\HC_G$.
Likewise there is a natural filtered equivalence $$\Mod_{\cD_\hbar(G)}\simeq\Mod_{\HC_{\hbar,G}}$$ between $\cD_\hbar(G)$ and the rigid filtered monoidal category $\HC_{\hbar,G}$. 
\end{theorem}

\begin{proof}
We sketch Beraldo's proof and its straightforward extension to the filtered setting. 
Given a $\cD(G)$-module we can take its weak invariants, i.e., invariants as a $\QC(G)$-module under the monoidal induction functor $\QC(G)\to \cD(G)$. This defines a continuous and conservative functor $\Mod_{\cD(G)}\to\Mod_{Rep(G)}$, a composition of the continuous and conservative forgetful functor to $\QC(G)$-modules with Gaitsgory's equivalence. The functor admits a left adjoint provided by 
$$\cE\mapsto \cD(G)\otimes_{\QC(G)} \Vect \otimes_{Rep(G)}\cE\simeq \Mod_{\fg}\otimes_{Rep(G)}\cE.$$ Applying the Barr-Beck-Lurie theorem, we identify $\Mod_{\cD(G)}$ with modules for the algebra object $\Mod_{\fg}^{(G)}\simeq \HC_G$ in $Rep(G)$-bimodules, from which the Morita equivalence follows. 

To see that $\HC_G$ is rigid, note that the compact objects in $Rep G$ are automatically left and right dualizable. Observe that push forward along the morphism $\pi:pt/G=G\bs G/G\to \GGG$ induces a compact and monoidal functor
$\pi_*:Rep G\to \HC_G$. (Here compactness follows from $\pi$ being inf-proper.) Thus the images of the compact, dualizable generators of $Rep G$ are compact and dualizable generators of $\HC_G$.

In the filtered version we likewise take weak invariants of a $\cD_\hbar(G)$-module, providing a continuous and conservative functor landing in $\Mod_{Rep(G)}(\Cat_{filt})$. It admits a left adjoint identified with 
$$\cE\mapsto \Mod_{\fg_\hbar}\otimes_{Rep(G)}\cE,$$ and the monad is identified with $\Mod_{\fg_\hbar}^{(G)}\simeq \HC_{\hbar,G}$. Rigidity of $\HC_{\hbar,G}$ follows analogously from the inf-properness of $pt/G \times \A^1/\Gm \to G\backslash G_{Hod}/G$.
\end{proof}

\begin{corollary}[\cite{dario}, Theorem 3.6.1\footnote{The theorem is stated for pro-unipotent groups, but the proof in the finite dimensional case applies to any affine algebraic group.}]
	There is a canonical equivalence
	\[
	S:\cC_G \xrightarrow{\sim} \cC^G
	\]
\end{corollary}
The above equivalence takes an explicit form: it is induced by the functor 
\begin{align*}
\cC \to \cC^G\\
c \mapsto \omega_G \ast c
\end{align*}
where $\omega_G = (G\to pt)^!(\C)$.
\begin{proof}(Sketch)
The corollary is a consequence of the rigidity of $\HC_G$ and the self-duality of $\Rep(G)\in \Cat$. Namely under the Morita equivalence the augmentation $\D(G)$-module $\Vect$ is identified with the $\HC_G$-module $\Rep(G)$, which is self-dual by rigidity of $\HC_G$ (see~\cite{1affine} or ~\cite{DGcat}). 
\end{proof}
Note that the proof has an obvious filtered analog.

Now, let us consider the case when $\cC = \cD(X)$, for some variety $X$ with an action of $G$. The quotient stack $X/G$ is the colimit of a simplicial variety:
\[
(X/G)_\bullet = X \leftleftarrows G\times X \ldots
\]
Let $\cD(X/G)^{\bullet,!}$ (respectively $\cD(X/G)_{\bullet,\ast}$) denote the cosimplicial (respectively, simplicial) object of $\Cat$ with simplices $\cD(X\times G^n)$, obtained by taking the upper shriek (respectively, lower star) functor associated to the structure maps of $(X/G)_\bullet$. 

\begin{lemma}
	The cosimplicial (respectively, simplicial) object $\cD(X/G)^{\bullet , !}$ (respectively, $\cD(X/G)_{\bullet, \ast}$) is canonically identified with Diagram \ref{diagraminvariants} (respectively, Diagram \ref{diagramcoinvariants}) with $\cC = \cD(X)$.
\end{lemma}
\begin{proof}
The theory of integral transforms for $D$-modules (see~\cite{GR} or~\cite{character}) gives equivalences
\[
\cD(G)^{\otimes n} \otimes \cD(X) \simeq \cD(G^n \times X) \simeq \Hom(\cD(G)^{\otimes n}, \cD(X))
\] (In fact these equivalences only rely on dualizability properties of the relevant categories, and thus extend to the filtered setting.)
This at least identifies the simplices of the diagrams in the Lemma. That the structure maps agree is immediate for the simplicial diagram, and for the cosimplicial diagram, we note that under the self duality of categories of the form $\cD(Y)$ for a smooth variety $Y$, the dual of a functor of the form $f_\ast$ is $f^!$.
\end{proof}

Let $\pi:X = (X/G)_0 \to X/G$ denote the quotient map.
\begin{proposition}\label{(co)inv}
There are equivalences
	\[
	\xymatrix{
	\cD(X)_G = \colim \cD(X/G)_{\bullet, \ast} \ar[r]_-{\pi_\ast}^-{\sim} & \cD(X/G) \ar[r]_-{\pi^!}^-{\sim} & \lim \cD(X/G)^{\bullet,!} = \cD(X)^G
}	
	\]
	\end{proposition}
\begin{proof}
	The equivalence on the right is the statement of smooth descent for $D$-modules (which extends to the filtered setting). Proposition~\ref{(co)inv} then gives an equivalence $\cD(X)_G \simeq \cD(X/G)$. Noting that the functor $\pi_! \pi_\ast: \cD(X) \to \cD(X)$ is identified with $M \mapsto \omega_G \ast M$ (via base-change) shows that this equivalence is indeed induced by the map $\pi_\ast$, as claimed.
	\end{proof} 
	
	\begin{remark}
		In general, colimits in $\Cat$ are hard to compute. In the case where the diagram $\cC_\bullet$ (of which we want to compute the colimit) involves maps which admit continuous right adjoints, then the colimit of $\cC_\bullet$ may be identified with the limit of $(\cC_\bullet)^R$ (the diagram obtained by taking right adjoints of all the structure maps). For example, this is the case when we consider bar constructions coming from modules for \emph{rigid} monoidal categories, as explained in \cite{1affine} (see also~\cite{character,morita}). Our situation is not of this form: the monoidal category $\cD(G)$ is not rigid (note that the structure maps of of $\cD(X/G)^{\bullet, !}$ are not right adjoint to those of $\cD(X/G)_{\bullet,\ast}$, but rather \emph{left} adjoint---at least, up to a shift). However, the monoidal category $\HC_G$ is rigid, and Theorem~\ref{Morita thm} allows one to compute relative tensor products over $\cD(G)$ (in particular, coinvariants) in terms of those over $HC$, leading to the proposition above.
	\end{remark}

\subsection{Center and Trace of $\cD(G)$}
Now consider the case of $G$ acting on itself by conjugation, with quotient stack denoted $G\adjquot G$. In this case, the diagrams \ref{diagramcoinvariants} and \ref{diagraminvariants} (or equivalently the diagrams $\cD(G\adjquot G)^{\bullet, !}$ and $\cD(G\adjquot G)_{\bullet, \ast}$) are naturally identified with the cyclic bar and cobar constructions computing $\Tr(\cD(G))$ and $\cZ(\cD(G))$ respectively.

Thus, Proposition \ref{(co)inv} gives canonical equivalences 
\begin{equation}\label{eq-centertrace}
\Tr(\cD(G)) \simeq \cD(G\adjquot G) \simeq Z(\cD(G)). 
\end{equation}
This allows us to define the trace map underlying the proposed Calabi-Yau structure on $\cD(G)$:
\begin{align*}
t = p_\ast e^!: \Tr(\cD(G)) \simeq \cD(G\adjquot G) \to \Vect \\
\end{align*}	
where 
\begin{equation}\label{trace diagram}
\xymatrix{
	G\adjquot G & \ar[l]_e pt/G \ar[r]^p & pt }
\end{equation}

\begin{lemma}\label{lemma-nd}
	The composite:
	\[
	\cD(G) \otimes \cD(G) \to \cD(G) \to \cD(G\adjquot G) \to \Vect
	\]
	is the evaluation map for a self duality of $\cD(G)$.
\end{lemma}
\begin{proof}
	By base change, the composite functor is given by $\pi_{G\ast} \nabla^!$, where $\nabla:G\to G\times G$,  $\nabla(g)= (g^{-1},g)$. The corresponding coevaluation map is given by $\pi_G^! \nabla_\ast$, and the standard duality identities (``Zorro conditions") follow from base change. 
\end{proof}
\begin{remark}
The self duality of $\cD(G)$ induced by the trace map differs from the usual self duality of $\cD(G)$ (coming from Verdier duality) by the automorphism $\sigma_\ast$, where $\sigma(g)=g^{-1}$.
\end{remark}

\subsection{Cyclic structure on class $D$-modules}
To show that the trace map $t$ defines a Calabi-Yau structure on $\cD(G)$ it remains to show the trace map defined above carries an $S^1$-equivariant structure. This requires a geometric reinterpretation of the above constructions, which we now discuss.
Note that $G\adjquot G = \cL(BG) = \Map(S^1,BG)$ carries a natural $S^1$-action, giving rise to an action of $S^1$ on $\cD(G\adjquot G)$. 
\begin{proposition}\label{prop-cyclic}
The equivalence of Equation~\ref{eq-centertrace} 
\[
\Tr(\cD(G)) \xrightarrow{\sim} \cD(G\adjquot G).
\]
carries a canonical $S^1$-equivariant structure.
\end{proposition}
\begin{proof}
To prove the proposition, we will exhibit the augmentation $\cD(G\adjquot G) \to \cD(G\adjquot G)_{\bullet, \ast}$ as the fiber of a family of such morphisms over the moduli space of framed circles $BS^1$. 

Recalling the notation of Section~\ref{cyclic section}, given a framed circle $S$ and an object $I$ of $I_S$, we associate the moduli space $\cM_G(S,I)$ of $G$-local systems on $S$, trivialized on the complement of the intervals. Let $\cM(S)$ denote the moduli of $G$-local systems on $S$. There are forgetful maps $\cM_G(S,I) \to \cM_G(S)$ for each $I$, compatible with the structure maps of $I_S$. Note that $\cM_G(S,I) \simeq G^{\pi_0(I)}$, and $\cM_G(S)) \simeq G\adjquot G$. The underlying simplicial set of the $I_S$-diagram $\cM_G(S,I)$ is equivalent to $(G\adjquot G)_\bullet$, and the forgetful map $\cM_G(S,I) \to \cM_G(S)$ induces the augmentation map $(G\adjquot G)_\bullet \to G\adjquot G$. Furthermore, the $I_S$-diagram of categories $\cD(\cM_G(S,I))$ (taking the lower star functor associated to the structure maps) is naturally identified with the cyclic bar construction of the category $\cD(G)$ as described e.g. in~\cite{character}.

In particular, the morphism
\[
\colim_{I_S} \cD(\cM_G(S,I)) \to \cD(\cM(S))
\]
carries a natural $S^1$-equivariant structure. The left hand side is $\Tr(\cD(G))$, the right hand side is $\cD(G\adjquot G)$, and the morphism is precisely the one constructed in Proposition \ref{(co)inv}. This defines the required $S^1$-equivariant structure.
\end{proof}

We also note:
\begin{lemma}\label{lemma-trace}
	The functor $t = p_\ast e^!$ carries an $S^1$-equivariant structure.
\end{lemma}
\begin{proof}
	As in Proposition~\ref{prop-cyclic}, the idea is to exhibit the diagram~\ref{trace diagram} as the fiber of a family of such diagrams over the moduli of framed circles. For such a circle $S$, consider the map $S\to pt$, which is naturally equivariant for the action of $S^1$ by rotations.\footnote{It is natural from the perspective of TFT to think of $pt$ as a disc, with $S$ being included as its boundary.} Thus, the diagram~\ref{trace diagram} is equivalent to:
	\[
	\cM_G(S) \leftarrow \cM_G(pt) \rightarrow \cM_G(\emptyset).
	\]
	from which the required $S^1$-equivariant structure follows.
\end{proof} 

\begin{proof}[Proof of Theorem \ref{thm-cy}]
	We have shown that $\cD(G)$ carries an $S^1$-equivariant trace map,$t$, (by Lemma \ref{lemma-trace} and Proposition \ref{prop-cyclic}), which is non-degenerate (by \ref{lemma-nd}), as required.
\end{proof}

\subsubsection{Monoidal structure on class $D$-modules}
We also note that the category $\cD(G\adjquot G)=\cD((pt/G)^{S^1})$ carries an $E_2$-monoidal structure arising from the pair of pants multiplication. The Drinfeld center $Z(\cA)$ of a monoidal category $\cA$ also carries such a structure. It is true that the functor 
\[
\pi^!:\cD(G\adjquot G) \xrightarrow{\sim} Z(\cD(G))
\]
defined above carries an $E_2$-monoidal structure. However, for the purposes of this paper it will be enough to note that it is just monoidal. In fact, we only need that it is naively monoidal, meaning that there exists an isomorphism
\[
\pi^!(M\ast N) \simeq \pi^!(M) \ast \pi^!(N)
\]
This follows from the observation that the comonadic forgetful functors from $\cD(G\adjquot G)$ and $Z(\cD(G))$ to $\cD(G)$ are both monoidal by construction.

\section{The character field theory}
In this section we will study the topological field theory $\cX_G$ defined, via the cobordism hypothesis, from the 2-dualizable Calabi-Yau category $\HC_G$. We also consider the classical version $\cX_{0,G}$ associated to classical Harish-Chandra bimodules $\HC_{0,G}$ and the filtered version.
Let $\Bord_2^{or}$ denote the $(\infty,2)$-category of oriented bordisms of 2-manifolds. 

\begin{definition} The {\em character TFT} $\cX_G$ is the symmetric monoidal functor $$\cX_G:\Bord_2\longrightarrow \Alg_{(1)}(\Cat)$$
defined by applying the Cobordism Hypothesis to the $SO(2)$-invariant 2-dualizable object $\HC_G$.
\end{definition}

 In particular we have for a closed oriented 1-manifold $$\cX_G(M)\simeq \cD(\Loc_G(M)).$$
 
\subsection{The character sheaf theory}
Let $\Bord_{1,2}^{or}$ denote the $(\infty,1)$-category of oriented bordisms of surfaces. We would like to give a geometric description of the functor
$$\cX_G^+:\Bord_{1,2}^{or}\longrightarrow \Cat$$ given by restricting the field theory $\cX_G$ to closed 1-manifolds.

Let us introduce another (1,2)-dimensional topological field theory, namely the functor given by the composition
$$\xymatrix{\cX_G^{+,geom}:\Bord_{1,2}^{or}\ar[r]^-{\Loc_G} & Corr \ar[r]^-{\cD} &\Cat}$$
of the functor from the bordism category to the correspondence category of stacks given by passing to stacks of local systems, i.e., $$\Loc_G(-)= [-, BG],$$ and the functor on the correspondence category given by $\cD$-modules \cite[III.4]{GR}. 

Let us spell out this functor. Let $S_i:\partial_{in} S_i\to \partial_{out} S_i$ be oriented surfaces with boundary considered as cobordisms. Given an identification $\partial_{out}S_1\simeq \partial_{in} S_2$, let $S$ denote the sewed surface $$S=S_1\coprod_{\partial_{out}S_1\simeq \partial_{in} S_2} S_2.$$ Then moduli of local systems on the surfaces compose as the composition of spans, i.e., as fiber products of stacks:

$$\xymatrix{&&\ar[dl] \Loc_G(S) \ar[dr]&& \\
& \ar[dl]\Loc_G(S_1)\ar[dr]&&\ar[dl]\Loc_G(S_2)\ar[dr]&\\
\Loc_G(\partial_{in}S_1) &&\Loc_G(\partial_{out}S_1\simeq \partial_{in} S_2) &&\Loc_G(\partial_{out}S_2)}$$

The functor $\cD$ is constructed in Section III.4 of~\cite{GR}. In fact they construct a symmetric monoidal category out of an $(\infty,2)$-category of prestacks locally almost of finite type, with morphisms given by correspondences which are ind-nil-schematic and 2-morphisms given by proper maps of correspondences. We will only need the restriction of this functor to the $(\infty,1)$-category in which we allow only isomorphisms of correspondences, and where the objects are geometric stacks of finite type (in fact smooth ones). 
To a stack $X$ the functor attaches the dg category $\cD(X)$ of $\cD$-modules on $X$, while to a correspondence 
$$\xymatrix{X & \ar[l]_-q Z \ar[r]^-p & Y}$$
it attaches the functor $$p_*q^!:\cD(X)\longrightarrow \cD(Y).$$ In particular for $X=Y=pt$ we have $$Z\mapsto p_*q^!\C=p_*\omega_Z=C_\ast^{BM}(Z)\in \Vect,$$ the cohomology of the dualizing sheaf, or Borel-Moore chains, of $Z$.

Thus we see that for a closed oriented 1-manifold the functor $\cX_G^{+,geom}$ attaches the category 
 $$\cX_G^{+,geom}(M)\simeq \cD(\Loc_G(M)) $$ -- concretely we have $$S^1\mapsto \cD(G\adjquot G).$$ Moreover for a closed surface $S$ we 
have  $$\cX_G^{+,geom}(S)\simeq C_\ast^{BM}(\Loc_G(S)),$$ the Borel-Moore chains of the character stack.

\subsection{Geometric identification of the character theory}

We will not give the full structure of a natural equivalence of the two functors $\cX_G^+$ and $\cX_G^{+,geom}$, but content ourselves with the existence of an unstructured identification of their values on any surface with boundary.

\begin{theorem}\label{geometric character}
For any oriented surface with boundary $S,\partial S= \partial_{in} S\coprod \partial_{out} S$ we have an equivalence
$$\xymatrix{\cX_G^{+}(S)\simeq \cX_G^{+,geom}:  \cD(\Loc_G(\partial_{in} S))\ar[r]& \cD(\Loc_G(\partial_{out} S))}.$$
\end{theorem}

\begin{proof}
First consider $S=D^2$ as a bordism from $S^1$ to the empty set. In $\cX_G^+$ this bordism defines the trace on $\Tr(\HC_G)\simeq \cD(G/G)$ prescribing the Calabi-Yau structure, and as such we've identified it with the functor given by the correspondence $$\xymatrix{\Loc_G(S^1)& \ar[l] \Loc_G(pt) \ar[r]& pt},$$ which agrees with $\cX_G^{+,geom}$. 

Next consider the bordism $D^2$ from the empty set to $S^1$ and a fixed pair of pants bordism from $S^1 \coprod S^1$ to $S^1$. In $\cX_G^+$ these bordisms are part of the $E_2$ monoidal structure on the center $\cZ(\HC_G)\simeq \cX_G^+(S^1)$, namely the unit and a binary product operation. We have identified $\cZ(\HC_G)$ with $\cD(G/G)=\cD(\Loc_G(S^1))$. We also checked the compatibility of this identification with with the unit and the binary multiplication map coming from convolution on $G$, i.e. of $G$-local systems on marked circles, hence with the $\cD$-module functors applied to bordisms from unions of $S^1$ to itself, in other words with the morphisms defining the corresponding operations in $\cX_G^{+,geom}$.

The theorem now follows from these basic building blocks using Proposition~\ref{basic blocks}.
\end{proof}

\begin{proposition}\label{basic blocks}
Let $\cZ,\cW$ denote two symmetric monoidal functors $$\cZ,\cW:Bord_{1,2}\longrightarrow \cC,$$ and assume given equivalences
$$\iota_{S^1}:\cZ(S^1)\simeq \cW(S^1)$$ as well as equivalences 
$$\iota_{M}:\cZ(M)\simeq \cW(M)$$ over $\iota_{S^1}$ for $M$ an incoming and outgoing disc and a pair of pants. 
Then for any oriented surface with boundary $S$ there exists an equivalence $\cZ(S)\simeq \cW(S)$.
\end{proposition}

\begin{proof}
Consider the pairing $\cZ(S^1)^{\otimes 2}\to 1_\cC$  given by the outgoing elbow, i.e., composition of the pair of pants and the outgoing disc. This pairing is the evaluation for a self-duality of $\cZ(S^1)$, with coevaluation given by the incoming elbow. Since the given data $\iota$ contain an identification of $\cZ$ and $\cW$ on the circle, pair of pants and outgoing discs, hence of the evaluation pairings, it follows that $\iota$ induces an identification of $\cZ(S^1)$ and $\cW(S^1)$ as self-dual objects. In particular we obtain an identification of $\cZ$ and $\cW$ on the incoming pair of pants.

Given a general surface with boundary $S,\partial S$ we first choose an identification of $\partial S$ with a disjoint union of standard circles, and thus an identification of $\cZ(\partial S)$ with $\cW(\partial S)$.
Next choose a decomposition of $S$ as a composition of morphisms in $Bord_{1,2}$ each of which is identified with an incoming or outgoing pair of pants or an incoming or outgoing disc (and respecting the given identification of $\partial S$). We have therefore expressed $\cZ(S)$ and $\cW(S)$ as compositions of morphisms between equivalent objects of $\cC$, which have also been identified. The proposition follows.
\end{proof}

\subsection{Filtered version}
The proof of Theorem~\ref{geometric character} extends directly to the filtered setting, given the functoriality of filtered $\cD$-module outlined in Section~\ref{filtered section}. In other words, we can relate the values of the restriction 
$$\cX_{\hbar,G}^+:\Bord_{1,2}^{or}\longrightarrow \Cat_{filt}$$ on morphisms with the values of the composition
$$\xymatrix{\cX_{\hbar,G}^{+,geom}:\Bord_{1,2}^{or}\ar[r]^-{\Loc_G} & Corr \ar[r]^-{\cD_{\hbar}} &\Cat_{filt}}.$$
The latter functor attaches to a stack $X$ the filtered category of sheaves on the Hodge stack $\cD_{\hbar}(X)=IndCoh(X_{Hod})$, 
and to a correspondence 
$$\xymatrix{X & \ar[l]_-q Z \ar[r]^-p & Y}$$
it attaches the functor $$p_*q^!:\cD_\hbar(X)\longrightarrow \cD_\hbar(Y).$$ In particular for $X=Y=pt$ we attach to $Z$ the filtered complex $\Gamma(Z_{Hod},\omega)\in \QC(\A^1/\Gm)$, the cohomology of the canonical sheaf on the deformation to the normal cone of the de Rham space of $Z$ - in other words (for $Z$ classical) with the naive Hodge filtration on de Rham cohomology defined by~\cite{SiT}. Note that in general we can identify this with the global sections of the Hodge filtration on $\omega_{Z_{dR}}$ itself defined in~\cite[IV.5.5]{GR}.

At $\hbar=0$, we can calculate the topological field theory $\cX_{0,G}$ directly as the ``commutative factorization homology" (tensoring of commutative algebras by simplicial sets) of the symmetric monoidal category $\HC_{0,G}=\QC(\fgx/G)$ of sheaves on the perfect stack $\fgx/G$. Hence it is easily evaluated (as in~\cite{BFN}) in terms of the mapping stacks $$[S,\fgx/G]=\{E\in \Loc_G(S), \; \eta\in \Gamma(ad(E)^*)\},$$
parameterizing local systems with a coadjoint section. Since deformations of local systems are given by cohomology of the adjoint representation, we find an identification (for $S$ a smooth compact $n$-manifold) with the shifted cotangent stack
$$[S,\fgx/G]\simeq T^*[1-n]\Loc_G(S).$$ In particular on $S^1$ we find $$\cX_{G,0}(S^1)\simeq \QC(T^*G\adjquot G)$$
and on a closed oriented surface we find $$\cX_{G,0}(S)\simeq \Gamma(T^*[-1]\Loc_G(S),\cO),$$ i.e. cohomology of polyvector fields (in homological degrees) on the character stack. This is precisely the global cohomology of the associated graded of the Hodge-filtration on the dualizing complex calculated by Gaitsgory-Rozenblyum~\cite[Proposition IV.5.5.3.2]{GR}.
For $G$ reductive we can use an invariant form on $\fg$ and the resulting symplectic form on the character stack to identify this also with functions on the shifted tangent bundle $T[-1]\Loc_G(S)$, i.e., differential forms.


\end{document}